\documentclass{article}
\usepackage{amsmath,amssymb,amsthm}

\usepackage[colorlinks=true,linkcolor=blue,citecolor=blue,pdfpagelabels=false]{hyperref}

\theoremstyle{plain}
  \newtheorem{theorem}{Theorem}[section]
  \newtheorem{lemma}[theorem]{Lemma}

\theoremstyle{definition}
  \newtheorem{example}[theorem]{Example}
  \newtheorem{remark}[theorem]{Remark}
  
\newenvironment{acknowledgements}{\bigskip\textbf{Acknowledgements.}}{}

\newcommand{\dueto}[1]{\textup{\textbf{(#1) }}}
\newcommand{\nin}{\not\in}
\newcommand{\nequiv}{\not\equiv}
\newcommand{\op}[1]{\ensuremath{\operatorname{#1}}}
\renewcommand{\pmod}[1]{\quad(\op{mod}\,#1)}

\renewcommand{\geq}{\geqslant}
\renewcommand{\leq}{\leqslant}

\begin{document}

\title{Congruences for Fishburn numbers modulo prime powers}

\author{Armin Straub
\thanks{Email: \texttt{astraub@illinois.edu}}\\
Department of Mathematics\\
University of Illinois at Urbana-Champaign}

% \date{\today}
\date{July 16, 2014}

\maketitle

\begin{abstract}
  The Fishburn numbers $\xi (n)$ are defined by the formal power series
  \begin{equation*}
    \sum_{n \geq 0} \xi (n) q^n = \sum_{n \geq 0} \prod_{j = 1}^n
     (1 - (1 - q)^j) .
  \end{equation*}
  Recently, G.~Andrews and J.~Sellers discovered congruences of the form $\xi
  (p m + j) \equiv 0$ modulo $p$, valid for all $m \geq 0$. These
  congruences have then been complemented and generalized to the case of
  $r$-Fishburn numbers by F.~Garvan. In this note, we answer a question of
  Andrews and Sellers regarding an extension of these congruences to the case
  of prime powers. We show that, under a certain condition, all these
  congruences indeed extend to hold modulo prime powers.
\end{abstract}

\section{Introduction and main results}

The Fishburn numbers $\xi (n)$, which are defined by the formal power series
\begin{equation}
  \sum_{n \geq 0} \xi (n) q^n = \sum_{n \geq 0} \prod_{j = 1}^n (1 -
  (1 - q)^j), \label{eq:fishburn}
\end{equation}
count a variety of combinatorial objects, such as interval orders or unlabeled
$( 2 + 2)$-free posets \cite{fishburn70}, upper-triangular matrices with
nonnegative integer entries of total sum $n$ and no zero rows or columns
\cite{dp-fishburn}, \cite{blr-fishburn}, regular linearized chord diagrams
or Stoimenow matchings \cite{stoimenow-fishburn}, ascent sequences
\cite{bcdk-fishburn}, or non-2-neighbor-nesting matchings
\cite{levande-fishburn}. The generating function \eqref{eq:fishburn} is due
to D.~Zagier \cite{zagier-strange}, who used it to derive an asymptotic
expansion for the Fishburn numbers. On the other hand, striking arithmetic
properties of these numbers, such as a connection with (mock) modular forms
\cite{zagier-strange}, \cite{bopr-unimodal} and congruences resembling
those of the partition function \cite{as-fish}, \cite{garvan-fish} have
recently been discovered. This note is concerned with generalizing these
congruences.

In order to state the congruences discovered by G.~Andrews and J.~Sellers,
define, as in \cite{as-fish}, the sets
\begin{equation*}
  S (p) = \{j \in \{0, 1, \ldots, p - 1\}: \hspace{1em} \text{$j \equiv
   \tfrac{1}{2} n (3 n - 1)$ for some $n$} \}
\end{equation*}
of residues modulo $p$ which are pentagonal numbers. The main result of
\cite{as-fish} is the following.

\begin{theorem}
  {\dueto{Andrews and Sellers, \cite{as-fish}}}\label{thm:as}Let $p$ be a
  prime. If $j \in \{1, 2, \ldots, p - 1 - \max S ( p) \}$, then
  \begin{equation}
    \xi ( p m - j) \equiv 0 \pmod{p}
    \label{eq:fishburn:c:as}
  \end{equation}
  for all positive integers $m$.
\end{theorem}

For primes $p$, which are quadratic nonresidues modulo $23$, the set $\{1, 2,
\ldots, p - 1 - \max S ( p) \}$ is nonempty; thus congruences of the form
\eqref{eq:fishburn:c:as} exist for infinitely many primes (indeed, half of
them). F.~Garvan proved \cite{garvan-fish} that such congruences also exist
for $p = 23$ (in which case $p - 1 - \max S ( p) = 0$) by replacing the set $S
( p)$ in Theorem~\ref{thm:as} with the smaller set $S^{\ast} ( p) = S ( p)
\backslash \{ i_0 \}$, where $i_0 \in \{ 0, 1, \ldots, p - 1 \}$ is such that
$i_0 \equiv - 1 / 24$ modulo $p$. The purpose of this note is to extend all of
these congruences to prime powers, thus answering a question posed by Andrews
and Sellers \cite{as-fish}. The following is a corollary of our main result,
which is stated as Theorem~\ref{thm:cong} below.

\begin{theorem}
  \label{thm:as:x}Let $p \geq 5$ be a prime. If $j \in \{1, 2, \ldots, p
  - 1 - \max S^{\ast} ( p) \}$, then
  \begin{equation*}
    \xi ( p^{\lambda} m - j) \equiv 0 \pmod{p^{\lambda}}
  \end{equation*}
  for all positive integers $m$ and $\lambda$.
\end{theorem}

\begin{example}
  \label{eg:1}Theorem~\ref{thm:as:x} includes the congruences
  \begin{eqnarray*}
    \xi ( 5^{\lambda} m - 1) \equiv \xi ( 5^{\lambda} m - 2) & \equiv & 0
    \pmod{5^{\lambda}},\\
    \xi ( 7^{\lambda} m - 1) & \equiv & 0 \pmod{7^{\lambda}},\\
    \xi ( 11^{\lambda} m - 1) \equiv \xi ( 11^{\lambda} m - 2) \equiv \xi (
    11^{\lambda} m - 3) & \equiv & 0 \pmod{11^{\lambda}},\\
    \xi ( 17^{\lambda} m - 1) & \equiv & 0 \pmod{17^{\lambda}},\\
    \xi ( 19^{\lambda} m - 1) \equiv \xi ( 19^{\lambda} m - 2) & \equiv & 0
    \pmod{19^{\lambda}},
  \end{eqnarray*}
  which generalize the ones obtained in \cite{as-fish}, as well as
  \begin{equation*}
    \xi ( 23^{\lambda} m - j) \equiv 0 \pmod{23^{\lambda}}
  \end{equation*}
  for $j \in \{ 1, 2, \ldots, 5 \}$, which generalize the additional
  congruences of \cite{garvan-fish}.
\end{example}

The congruences for the Fishburn numbers $\xi ( n)$ provided by
Theorem~\ref{thm:as:x} are a special case of a more general result concerning
the numbers $\xi_{r, s} ( n)$ defined by
\begin{equation}
  \sum_{n \geq 0} \xi_{r, s} (n) q^n = ( 1 - q)^s \sum_{n \geq 0}
  \prod_{j = 1}^n (1 - (1 - q)^{r j}) = ( 1 - q)^s F ( ( 1 - q)^r),
  \label{eq:fishburn:rs}
\end{equation}
where $F ( q)$ is defined in \eqref{eq:F}. Clearly, the case $r = 1$, $s = 0$
reduces to the Fishburn numbers $\xi ( n) = \xi_{1, 0} ( n)$. The case $r = -
1$ has combinatorial significance as well. The numbers $( - 1)^n \xi_{- 1, 0}
( n)$ count, for instance, upper-triangular matrices with $0$-1-entries of
total sum $n$ and no zero rows or columns \cite{aj-fishburn},
\cite{blr-fishburn}. In the case $s = 0$, we obtain the $r$-Fishburn numbers
$\xi_r ( n) = \xi_{r, 0} ( n)$ introduced by Garvan \cite{garvan-fish}. Note
that, for positive integers $s$, the numbers
\begin{equation*}
  \xi_{r, s} ( n) = \sum_{j = 0}^s \binom{s}{j} ( - 1)^j \xi_r ( n - j)
\end{equation*}
are linear combinations of $r$-Fishburn numbers. As mentioned in
\cite{garvan-fish}, congruences satisfied by such combinations, such as
\begin{equation}
  \xi ( 5 m + 2) - 2 \xi ( 5 m + 1) \equiv 0 \pmod{5},
  \label{eq:gr5}
\end{equation}
have been first observed by S.~Garthwaite and R.~Rhoades. In order to state
the main result of \cite{garvan-fish}, let us introduce the sets
\begin{equation*}
  S (p, r, s) = \{j \in \{0, 1, \ldots, p - 1\}: \hspace{1em} \text{$j - s
   \equiv \tfrac{1}{2} r n (3 n - 1)$ for some $n$} \}
\end{equation*}
as well as, for $p \geq 5$, the smaller sets $S^{\ast} (p, r, s) = S (p,
r, s) \backslash \{ i_0 \}$, where $i_0 \in \{ 1, 2, \ldots, p - 1 \}$ is such
that $i_0 - s \equiv - r / 24$ modulo $p$. In other words, modulo $p$, the
sets $S^{\ast} (p, r, s) \equiv r S^{\ast} ( p) + s$ are linear translations
of the sets $S^{\ast} ( p)$.

\begin{theorem}
  {\dueto{Garvan, \cite{garvan-fish}}}\label{thm:g}Let $p \geq 5$ be a
  prime, $r$ an integer such that $p \nmid r$, and $s \in \{ 0, 1, \ldots, p -
  1 \}$. If $j \in \{1, 2, \ldots, p - 1 - \max S^{\ast} ( p, r, s) \}$, then
  \begin{equation}
    \xi_{r, s} ( p m - j) \equiv 0 \pmod{p}
    \label{eq:fishburn:c:g}
  \end{equation}
  for all positive integers $m$.
\end{theorem}

Our main result generalizes these congruences to prime powers. However, as is
illustrated by Examples~\ref{eg:523} and \ref{eg:garvan5} below, not all of
these congruences extend to congruences modulo $p^{\lambda}$, so that some
restriction on $p$, $r$ and $s$ is necessary. Recall that every rational
number $n$ has a unique $p$-adic expansion
\begin{equation}
  n = \sum_{k = \nu_p ( n)}^{\infty} n_k p^k \label{eq:padic}
\end{equation}
where $n_k \in \{ 0, 1, \ldots, p - 1 \}$ and $\nu_p ( n)$ is the $p$-adic
valuation of $n$. We write $\operatorname{digit}_k ( n ; p) = n_k$. The following is
our main result.

\begin{theorem}
  \label{thm:cong}Let $p$ be a prime, and $r, s$ integers such that $p \nmid
  r$. If $j \in \{1, 2, \ldots, p - 1 - \max S ( p, r, s) \}$, then
  \begin{equation*}
    \xi_{r, s} ( p^{\lambda} m - j) \equiv 0 \pmod{p^{\lambda}}
  \end{equation*}
  for all positive integers $m$ and $\lambda$. Moreover, if $p \geq 5$
  and the triple $( p, r, s)$ satisfies
  \begin{equation}
    \operatorname{digit}_1 ( s - r / 24 ; p) \neq p - 1, \label{eq:cong:C}
  \end{equation}
  then the set $S ( p, r, s)$ may be replaced with $S^{\ast} ( p, r, s)$.
\end{theorem}

Note that, if $r = 1$ and $s = 0$, then condition \eqref{eq:cong:C} is
satisfied for all primes $p \geq 5$, as can be seen from $( p^2 - 1) / 24
\in \mathbb{Z}$ and
\begin{equation*}
  \operatorname{digit}_1 ( - 1 / 24 ; p) = \left\lfloor \frac{p^2 - 1}{24 p}
   \right\rfloor = \left\lfloor \frac{p}{24} \right\rfloor < p - 1.
\end{equation*}
Hence, Theorem~\ref{thm:as:x} is a corollary of Theorem~\ref{thm:cong}.

Our proof of Theorem~\ref{thm:cong} is given in Section~\ref{sec:proof}. It
relies crucially on results of Andrews, Sellers and Garvan, which are recalled
in Section~\ref{sec:prelim}, as well as careful applications of Kummer's
theorem on $p$-adic valuations of binomial coefficients. For most of the
remainder of this introduction, we illustrate Theorem~\ref{thm:cong} with a
number of examples.

\begin{example}
  \label{eg:gr}In the case $p = 5$ and $r = 1$, Theorem~\ref{thm:cong} shows
  that, in addition to
  \begin{equation*}
    \xi ( 5^{\lambda} m - 1) \equiv \xi ( 5^{\lambda} m - 2) \equiv 0
     \pmod{5^{\lambda}},
  \end{equation*}
  which were already observed in Example \ref{eg:1}, we have the congruences
  \begin{equation*}
    \xi_{1, 3} ( 5^{\lambda} m - 1) \equiv 0 \pmod{5^{\lambda}} .
  \end{equation*}
  Combining these with $\xi_{1, 3} ( n) = \xi ( n) - 3 \xi ( n - 1) + 3 \xi (
  n - 2) - \xi ( n - 3)$, we conclude that
  \begin{equation*}
    \xi ( 5^{\lambda} m - 3) - 2 \xi ( 5^{\lambda} m - 4) \equiv 0
     \pmod{5^{\lambda}},
  \end{equation*}
  generalizing the congruence \eqref{eq:gr5}, which was observed by Garthwaite
  and Rhoades.
\end{example}

\begin{example}
  We have already seen that, if $r = 1$ and $s = 0$, then the condition
  \eqref{eq:cong:C} in Theorem~\ref{thm:cong} is automatically satisfied for
  all primes $p \geq 5$. Let us illustrate in this example that condition
  \eqref{eq:cong:C} is similarly vacuous for other ``small'' values of $r$ and
  $s$. For instance, in the particularly interesting case $s = 0$, we find
  that, if $r \in \{ 1, 2, \ldots, 23 \}$, then the only tuples $( p, r, s)$,
  for which failure of condition \eqref{eq:cong:C} prevents additional
  congruences, are $( 5, 23, 0)$ and $( 7, 23, 0)$. These two cases are
  discussed in the next example, which demonstrates that the respective
  congruences indeed fail to extend to prime powers.
  
  To see that in the remaining cases, for which $r \in \{ 1, 2, \ldots, 23 \}$
  and $s = 0$, we can ignore condition \eqref{eq:cong:C} in
  Theorem~\ref{thm:cong}, observe that, modulo $p^2$,
  \begin{equation*}
    s - r / 24 \equiv r \frac{p^2 - 1}{24}
  \end{equation*}
  and that the right-hand is an integer in $\{ 0, 1, \ldots, p^2 - 1 \}$.
  Condition \eqref{eq:cong:C} is therefore equivalent to the inequality $r (
  p^2 - 1) / 24 < p ( p - 1)$. Rewriting this as $( 24 - r) p > r$, we find
  that this inequality holds whenever $p > 23 \geq r$. Checking the
  remaining small primes individually for all $r \in \{ 1, 2, \ldots, 23 \}$,
  we conclude that the only cases, for which condition \eqref{eq:cong:C} fails
  and $\max S ( p, r, s) > \max S^{\ast} ( p, r, s)$, are the two cases $( 5,
  23, 0)$ and $( 7, 23, 0)$ mentioned above.
  
  A similar analysis can be carried out for other small values of $s$.
\end{example}

\begin{example}
  \label{eg:523}Note that $S ( 5, 23, 0) = \{ 0, 1, 3 \}$ and $S^{\ast} ( 5,
  23, 0) = \{ 0, 1 \}$, as well as $S ( 7, 23, 0) = \{ 0, 2, 3, 4 \}$ and
  $S^{\ast} ( 7, 23, 0) = \{ 0, 2, 3 \}$. Hence, by Garvan's
  Theorem~\ref{thm:g}, the congruences
  \begin{equation*}
    \xi_{23} ( 5 m - j) \equiv 0 \pmod{5}, \hspace{2em}
     \xi_{23} ( 7 m - j) \equiv 0 \pmod{7}
  \end{equation*}
  hold for all positive integers $m$ if $j \in \{ 1, 2, 3 \}$. On the other
  hand, since condition \eqref{eq:cong:C} is not satisfied,
  Theorem~\ref{thm:cong} shows that of these six families of congruences only
  three extend to prime powers; namely,
  \begin{equation*}
    \xi_{23} ( 5^{\lambda} m - 1) \equiv 0 \pmod{5^{\lambda}} , \hspace{2em} \xi_{23} ( 7^{\lambda} m - 2) \equiv
     \xi_{23} ( 7^{\lambda} m - 1) \equiv 0 \pmod{7^{\lambda}} .
  \end{equation*}
  On the other hand, the other three families cannot be extend to prime
  powers, as is demonstrated by
  \begin{equation*}
    \xi_{23} ( 5^2 - 3) \equiv 10 \nequiv 0 \pmod{5^2},
     \hspace{1em} \xi_{23} ( 5^2 - 2) \equiv 5 \nequiv 0 \pmod{5^2},
  \end{equation*}
  and $\xi_{23} ( 7^2 - 3) \equiv 42 \nequiv 0$ modulo $7^2$. We should
  mention that these numbers are already quite large. For instance,
  \begin{equation*}
    \xi_{23} ( 5^2 - 3) = 105368264798040017097834938676731639668933422960.
  \end{equation*}
\end{example}

\begin{example}
  \label{eg:garvan5}As mentioned above, the case $r = - 1$ has combinatorial
  significance as well. Following \cite{garvan-fish}, we observe that $1$ is
  a pentagonal number, which implies $p - 1 \in S ( p, - 1, 0)$. In fact, for
  all primes $p > 5$, $p - 1 \in S^{\ast} ( p, - 1, 0)$, so that no
  congruences are implied by Garvan's Theorem~\ref{thm:g} for these values of
  $p, r, s$. In the case $p = 5$, on the other hand, $S ( 5, - 1, 0) = \{ 0,
  3, 4 \}$ while $S^{\ast} (5, - 1, 0) = \{ 0, 3 \}$. Theorem~\ref{thm:g}
  therefore yields the congruences
  \begin{equation*}
    \xi_{- 1} ( 5 m - 1) \equiv 0 \pmod{5},
  \end{equation*}
  which were conjectured in \cite{as-fish}. However, condition
  \eqref{eq:cong:C} is not satisfied for $( p, r, s) = ( 5, - 1, 0)$ since
  $\operatorname{digit}_1 ( 1 / 24 ; 5) = 4$, so that Theorem~\ref{thm:cong} does not
  provide an extension of this congruence to prime powers $5^{\lambda}$.
  Indeed, the congruences do not extend, as is demonstrated by
  \begin{equation*}
    \xi_{- 1} ( 5^2 - 1) = 11115833059268126770 \equiv 20 \nequiv 0
     \pmod{5^2} .
  \end{equation*}
  This is another illustration that, for $\lambda > 1$, we cannot, in general,
  replace the set $S ( p, r, s)$ with the set $S^{\ast} (p, r, s)$ in
  Theorem~\ref{thm:cong}.
  
  On the other hand, we note that $S^{\ast} (5, - 1, 2) = \{ 0, 2 \}$ and that
  condition \eqref{eq:cong:C} is satisfied for $( p, r, s) = ( 5, - 1, 2)$.
  Hence, for $j \in \{ 1, 2 \}$ and all positive integers $m$ and $\lambda$,
  $\xi_{- 1, 2} ( 5^{\lambda} m - j) \equiv 0$ modulo $5^{\lambda}$, or,
  equivalently,
  \begin{eqnarray*}
    \xi_{- 1} ( 5^{\lambda} m - 1) - 2 \xi_{- 1} ( 5^{\lambda} m - 2) + \xi_{-
    1} ( 5^{\lambda} m - 3) & \equiv & 0 \pmod{5^{\lambda}},\\
    \xi_{- 1} ( 5^{\lambda} m - 2) - 2 \xi_{- 1} ( 5^{\lambda} m - 3) + \xi_{-
    1} ( 5^{\lambda} m - 4) & \equiv & 0 \pmod{5^{\lambda}}
    .
  \end{eqnarray*}
\end{example}

Finally, we complement the congruences of Theorem~\ref{thm:cong} with
additional congruences for the case when $p$ divides $r$. These congruences
are considerably simpler to prove and are provided here only for the sake of
completeness and context.

\begin{lemma}
  \label{lem:cong:pr}Let $p$ be a prime, and $r$ a nonzero integer. If $j$ is
  coprime to $p$, then
  \begin{equation*}
    \xi_{p^{\lambda} r} ( p m - j) \equiv 0 \pmod{p^{\lambda}}
  \end{equation*}
  for all positive integers $m$ and $\lambda$.
\end{lemma}

\section{Preliminaries}\label{sec:prelim}

The formal generating series \eqref{eq:fishburn} of the Fishburn numbers
equals $F (1 - q)$, where
\begin{equation}
  F (q) = \sum_{n \geq 0} (q ; q)_n \label{eq:F}
\end{equation}
is ``Kontsevich's strange function'' \cite{zagier-strange}. Note that the
series $F (q)$ is peculiar in that it does not converge in any open subset of
the complex plane; it is, however, well-defined, since it terminates, whenever
$q$ is a root of unity. As in \cite{as-fish}, we consider the truncated
series
\begin{equation*}
  F (q, N) = \sum_{n = 0}^N (q ; q)_n,
\end{equation*}
which we dissect as
\begin{equation}
  F (q, N) = \sum_{i = 0}^{p - 1} q^i A_p (N, i, q^p), \label{eq:Ap}
\end{equation}
where $A_p (N, i, q)$ is a polynomial in $q$. We record the following crucial
divisibility property of these polynomials, which was proved in
\cite{as-fish}.

\begin{lemma}
  \label{lem:alpha}Let $p$ be a prime. If $i \nin S (p)$, then
  \begin{equation*}
    A_p (p n - 1, i, q) = (1 - q)^n \alpha_p (n, i, q),
  \end{equation*}
  where the $\alpha_p (n, i, q)$ are polynomials in $q$ with integer
  coefficients.
\end{lemma}

As observed in \cite{as-fish}, numerical evidence suggests that, for $i \nin
S (p)$, the polynomials $A_p (p n - 1, i, q)$ are in fact divisible by $(q ;
q)_n$.

A similar result is shown in \cite{garvan-fish} for indices $i$ such that $i
\equiv - 1 / 24$ modulo a prime coprime to $24$. Since, modulo $p$, one
obtains the conclusion of Lemma~\ref{lem:alpha}, this result allowed Garvan to
prove the additional congruences provided by Theorem~\ref{thm:g}.

\begin{lemma}
  \label{lem:alpha24}Let $p \geq 5$ and $i \in \{ 1, 2, \ldots, p - 1
  \}$. If $i \equiv - 1 / 24$ modulo $p$, then
  \begin{equation*}
    A_p (p n - 1, i, q) = \left( \frac{12}{p} \right) p q^{\lfloor p / 24
     \rfloor} F ( q^p, p n - 1) + (1 - q)^n \beta_p (n, i, q),
  \end{equation*}
  where the $\beta_p (n, i, q)$ are polynomials in $q$ with integer
  coefficients.
\end{lemma}

\section{Proofs}\label{sec:proof}

Before turning to the proof of Theorem~\ref{thm:cong}, we prove some lemmas,
which will be helpful for establishing the desired congruences modulo
$p^{\lambda}$. In the sequel, given formal power series $A ( q), B ( q) \in
\mathbb{Z} [ [ q]]$, we write
\begin{equation*}
  A ( q) \equiv B ( q) \pmod{p^{\lambda}}
\end{equation*}
to mean that the coefficients of $A ( q)$ and $B ( q)$ agree modulo
$p^{\lambda}$.

To begin with, we recall a well-known result of Kummer \cite{kummer-binom},
which gives the $p$-adic valuation of binomial coefficients. As usual, and as
in \eqref{eq:padic}, the $p$-adic digits $n_0, n_1, n_2, \ldots$ of a
(possibly negative) integer $n$ are determined by $n_i \in \{ 0, 1, \ldots, p
- 1 \}$ and $n = n_0 + n_1 p + n_2 p^2 + \cdots$.

\begin{theorem}
  \label{thm:kummer}{\dueto{Kummer, \cite{kummer-binom}}}Let $p$ be a prime,
  and $n, k$ integers such that $k \geq 0$. Then the $p$-adic valuation
  of the binomial coefficient $\binom{n}{k}$ is equal to the number of carries
  when adding $k$ and $n - k$ in base $p$.
\end{theorem}

\begin{proof}
  Since this result is commonly only stated for the case $n \geq k
  \geq 0$, we assume that this case has been proved and show here only
  how to use it to deduce the case of general $n$. First, in the case $n
  \geq 0$ and $n < k$, the binomial coefficient is zero, so that its
  $p$-adic valuation is $\infty$. On the other hand, the $p$-adic addition of
  $k$ and $n - k < 0$ results in infinitely many carries.
  
  Suppose now that $n < 0$. Observe that, for large enough $\lambda$,
  \begin{equation*}
    \nu_p \left( \binom{n}{k} \right) = \nu_p \left( \binom{p^{\lambda} +
     n}{k} \right)
  \end{equation*}
  and, by Kummer's theorem in the nonnegative case, the right-hand side is the
  number of carries when adding $k$ and $p^{\lambda} + n - k$ in base $p$. It
  only remains to observe that this equals the number of carries when adding
  $k$ and $n - k$.
\end{proof}

\begin{lemma}
  \label{lem:cong:pr2}Let $p$ be a prime, $r$ an integer, and $j \in \{ 1, 2,
  \ldots, p - 1 \}$. If $s$ is an integer such that
  \begin{equation}
    \binom{s}{p^2 - j} \equiv 0 \pmod{p},
    \label{eq:cong:pr2:assum}
  \end{equation}
  then, for any positive integers $m$, $\lambda$,
  \begin{equation*}
    \xi_{p^2 r, s} ( p^{\lambda} m - j) \equiv 0 \pmod{p^{\lambda - 1}} .
  \end{equation*}
\end{lemma}

\begin{proof}
  Using Zagier's generating function \eqref{eq:fishburn}, we have
  \begin{equation*}
    \sum_{n \geq 0} \xi_{p^2 r, s} (n) q^n = ( 1 - q)^s \sum_{n
     \geq 0} \prod_{j = 1}^n (1 - (1 - q)^{p^2 r j}) .
  \end{equation*}
  Expanding the product on the right-hand side, this can be written in the
  form
  \begin{equation*}
    \sum_{n \geq 0} c_n ( 1 - q)^{p^2 r n + s}
  \end{equation*}
  for some integer coefficients $c_n$. It therefore suffices to show that, for
  any integer $a$, the coefficient of $q^{p^{\lambda} m - j}$ in $( 1 -
  q)^{p^2 a + s}$ vanishes modulo $p^{\lambda - 1}$. Equivalently, it is
  enough to prove that, for any integer $a$,
  \begin{equation}
    \binom{p^2 a + s}{p^{\lambda} m - j} \equiv 0 \pmod{p^{\lambda - 1}} . \label{eq:cong:pr2:binom}
  \end{equation}
  By Kummer's Theorem~\ref{thm:kummer}, the $p$-adic valuation of the binomial
  coefficient $\binom{n}{k}$ is equal to the number of carries when adding $k$
  and $n - k$ in base $p$. Write $n = p^2 a + s$, $k = p^{\lambda} m - j$ and
  $k' = n - k$. Denote with $k_0, k_1, \ldots$ the $p$-adic digits of $k =
  p^{\lambda} m - j$, and, likewise, with $k'_0, k_1', \ldots$ the digits of
  $k'$. By construction, $k_0 = p - j$ and $k_1 = k_2 = \ldots = k_{\lambda -
  1} = p - 1$. Since the digits $k_1, k_2, \ldots, k_{\lambda - 1}$ have
  maximal value, if a carry occurs when adding $k_0 + k_1 p$ and $k'_0 + k_1'
  p$, then the number of carries when adding $k$ and $k'$ is at least $\lambda
  - 1$. Hence, \eqref{eq:cong:pr2:binom} follows from Kummer's theorem.
  
  It remains to observe that, again by Kummer's Theorem~\ref{thm:kummer},
  assumption \eqref{eq:cong:pr2:assum} is equivalent to a carry occuring when
  adding $k_0 + k_1 p$ and $k'_0 + k_1' p$.
\end{proof}

\begin{lemma}
  \label{lem:bigO}Let $\lambda \geq 1$ be an integer. For integers $r$
  and $n$ such that $n \geq \lambda$,
  \begin{equation*}
    ( 1 - (1 - q)^{r p})^n \equiv O (q^{p n - (p - 1) (\lambda - 1)})
     \pmod{p^{\lambda}} .
  \end{equation*}
\end{lemma}

\begin{proof}
  Note that the case $\lambda = 1$ is an immediate consequence of the fact
  that $( 1 - q)^p \equiv 1 - q^p$ modulo $p$. In general, we conclude that
  $(1 - q)^p = 1 + p q f (q) - q^p$ for some polynomial $f (q) \in \mathbb{Z}
  [q]$. It follows from another binomial expansion that, for any integer $r$,
  \begin{equation*}
    1 - (1 - q)^{r p} = 1 - ( 1 + p q f (q) - q^p)^r = p q h_1 (q) + q^p h_2
     ( q)
  \end{equation*}
  for some power series $h_1 (q), h_2 ( q) \in \mathbb{Z} [ [ q]]$.
  Therefore, modulo $p^{\lambda}$,
  \begin{eqnarray*}
    ( 1 - (1 - q)^{r p})^n & = & \sum_{k = 0}^n \binom{n}{k} (p q h_1 (q))^k
    (q^p h_2 ( q))^{n - k}\\
    & \equiv & \sum_{k = 0}^{\lambda - 1} \binom{n}{k} (p q h_1 (q))^k (q^p
    h_2 ( q))^{n - k}\\
    & = & O (q^{(\lambda - 1) + p (n - (\lambda - 1))}),
  \end{eqnarray*}
  as claimed.
\end{proof}

\begin{lemma}
  \label{lem:binom0}Let $p$ be a prime, $a$ an integer and $i \in \{0, 1,
  \ldots, p - 1\}$. If $j$ is an integer such that $0 < j < p - i$, then
  \begin{equation*}
    \binom{p a + i}{p^{\lambda} m - j} \equiv 0 \pmod{p^{\lambda}} .
  \end{equation*}
\end{lemma}

\begin{proof}
  Write $n = p a + i$ and $k = p^{\lambda} m - j$. By Kummer's
  Theorem~\ref{thm:kummer}, we need to show that there are at least $\lambda$
  carries when adding $k$ and $n - k$ in base $p$. Let $n_0, n_1, \ldots$ be
  the $p$-adic digits of $n$, and let $k_0, k_1, \ldots$ be the $p$-adic
  digits of $k$. By assumption, $n_0 = i$ and $k_0 = p - j$, so that the
  condition on $j$ implies that $k_0 > n_0$. In other words, a carry occurs
  when adding the least-order $p$-adic digits of $k$ and $n - k$. Moreover,
  since $k_1 = k_2 = \ldots = k_{\lambda - 1} = p - 1$, we conclude that
  adding $k$ and $n - k$ in base $p$ results in at least $\lambda$ carries.
\end{proof}

We are now in a position to prove Theorem~\ref{thm:cong}.

\begin{proof}[Proof of Theorem~\ref{thm:cong}]
  We need to show that the coefficient of
  $q^{p^{\lambda} m - j}$ in
  \begin{equation*}
    ( 1 - q)^s F ( ( 1 - q)^r) = \lim_{N \rightarrow \infty} ( 1 - q)^s F ( (
     1 - q)^r, N)
  \end{equation*}
  vanishes modulo $p^{\lambda}$. Let $n \geq \lambda$ be an integer.
  Following \cite{as-fish}, we split the truncated series
  \begin{equation*}
    F ((1 - q)^r, p n - 1) = \sum_{i = 0}^{p - 1} (1 - q)^{r i} A_p (p n - 1,
     i, (1 - q)^{r p})
  \end{equation*}
  into the part corresponding to indices $i \in S (p)$ and a second part
  corresponding to indices $i \nin S (p)$. Here, the polynomials $A_p$ are as
  defined by \eqref{eq:Ap}. If $i \nin S (p)$, then Lemma~\ref{lem:alpha}
  implies that
  \begin{equation*}
    A_p (p n - 1, i, (1 - q)^{r p}) = (1 - (1 - q)^{r p})^n \alpha_p (n, i,
     (1 - q)^{r p}) .
  \end{equation*}
  It therefore follows from Lemma~\ref{lem:bigO} that, for $i \nin S (p)$,
  \begin{equation*}
    A_p (p n - 1, i, (1 - q)^{r p}) \equiv O (q^{p n - (p - 1) (\lambda -
     1)}) \pmod{p^{\lambda}} .
  \end{equation*}
  (In fact, it suffices to observe the simpler fact that $A_p (p n - 1, i, (1
  - q)^{r p}) = O (q^n)$.) We conclude that, modulo $p^{\lambda}$,
  \begin{equation*}
    F ((1 - q)^r, p n - 1) \equiv \sum_{i \in S (p)} (1 - q)^{r i} A_p (p n -
     1, i, (1 - q)^{r p}) + O (q^{p n - (p - 1) (\lambda - 1)}) .
  \end{equation*}
  Choosing $n$ large enough, it therefore suffices to show that the
  coefficient of $q^{p^{\lambda} m - j}$ in
  \begin{equation}
    (1 - q)^{r i + s} A_p (p n - 1, i, (1 - q)^{r p}) \label{eq:A:pf}
  \end{equation}
  vanishes modulo $p^{\lambda}$ for all $i \in S ( p)$.
  
  First, suppose that $j \in \{1, 2, \ldots, p - 1 - \max S ( p, r, s) \}$.
  Since $A_p (p n - 1, i, q)$ is a polynomial in $q$, it suffices to show
  that, for any integer $a$,
  \begin{equation}
    [q^{p^{\lambda} m - j}] (1 - q)^{a p + r i + s} \equiv 0 \pmod{p^{\lambda}} . \label{eq:coeffp0}
  \end{equation}
  Since $r$ is relatively prime to $p$, the condition $i \in S (p)$ implies $r
  i + s \equiv i'$ modulo $p$ for some $i' \in S (p, r, s)$. The condition on
  $j$ implies that $j < p - i'$, and congruence \eqref{eq:coeffp0} follows
  from Lemma~\ref{lem:binom0}.
  
  Finally, suppose that the triple $( p, r, s)$ satisfies condition
  \eqref{eq:cong:C} and that $j \in \{1, 2, \ldots, p - 1 - \max S^{\ast} ( p,
  r, s) \}$. Let $i_0 \in \{ 0, 1, \ldots, p - 1 \}$ be such that $i_0 \equiv
  - 1 / 24$ modulo $p$. If we can show that, for $i = i_0$, the coefficient of
  $q^{p^{\lambda} m - j}$ in \eqref{eq:A:pf} vanishes modulo $p^{\lambda}$,
  then it only remains to consider \eqref{eq:A:pf} for $i \in S^{\ast} ( p) =
  S ( p) \backslash \{ i_0 \}$ and the same argument as in the previous case
  proves our desired congruence.
  
  By Lemma~\ref{lem:alpha24},
  \begin{equation*}
    A_p (p n - 1, i_0, q) = \left( \frac{12}{p} \right) p q^{\lfloor p / 24
     \rfloor} F ( q^p, p n - 1) + O ( ( 1 - q)^n) .
  \end{equation*}
  As above, it follows from Lemma~\ref{lem:bigO} that, if $n$ is large enough,
  then we may ignore the term $O ( ( 1 - q)^n)$. It therefore suffices to show
  that the coefficient of $q^{p^{\lambda} m - j}$ in
  \begin{equation*}
    \left( \frac{12}{p} \right) p (1 - q)^{r i_0 + s + r p \lfloor p / 24
     \rfloor} F ( (1 - q)^{r p^2})
  \end{equation*}
  vanishes modulo $p^{\lambda}$. Equivalently, we need to prove that
  \begin{equation*}
    \xi_{p^2 r, r i_0 + s + r p \lfloor p / 24 \rfloor} ( p^{\lambda} m - j)
     \equiv 0 \pmod{p^{\lambda - 1}} .
  \end{equation*}
  These congruences follow from Lemma~\ref{lem:cong:pr2} if
  \begin{equation}
    \binom{r i_0 + s + r p \lfloor p / 24 \rfloor}{p^2 - j} \equiv 0
    \pmod{p} . \label{eq:cong:Cj}
  \end{equation}
  Observe that the $p$-adic expansion of $- 1 / 24$ is $i_0 + i_1 p + i_2 p^2
  + \cdots$, where $i_1 = \lfloor p / 24 \rfloor$. This is easily seen from
  \begin{equation*}
    \operatorname{digit}_1 \left( - \frac{1}{24} ; p \right) = \operatorname{digit}_1 \left(
     \frac{p^2 - 1}{24} ; p \right) = \left( \frac{p^2 - 1}{24} - i_0 \right)
     / p = \left\lfloor \frac{p^2 - 1}{24 p} \right\rfloor = \left\lfloor
     \frac{p}{24} \right\rfloor .
  \end{equation*}
  Since $p \nmid r$, it follows that
  \begin{equation}
    r i_0 + s + r p \lfloor p / 24 \rfloor \equiv s - r / 24 \pmod{p^2} . \label{eq:cong:i0}
  \end{equation}
  Combining \eqref{eq:cong:Cj} and \eqref{eq:cong:i0}, it only remains to
  prove
  \begin{equation}
    \binom{s - r / 24}{p^2 - j} \equiv 0 \pmod{p} .
    \label{eq:cong:Cj2}
  \end{equation}
  Observing that $\operatorname{digit}_1 ( p^2 - j ; p) = p - 1$, we find that this
  congruence follows from Kummer's Theorem~\ref{thm:kummer} and the assumption
  that $\operatorname{digit}_1 ( s - r / 24 ; p) < p - 1$. This completes our proof.
\end{proof}

\begin{remark}
  By Kummer's Theorem~\ref{thm:kummer}, congruence \eqref{eq:cong:Cj2} is
  equivalent to
  \begin{equation}
    \operatorname{digit}_1 ( s - r / 24 ; p) < p - 1, \hspace{1em} \text{or}
    \hspace{1em} \operatorname{digit}_0 ( s - r / 24 ; p) < p - j. \label{eq:cong:Cjx}
  \end{equation}
  The purpose of this remark is to note that the second of these conditions
  does not hold in the cases which are considered in the course of our proof
  of Theorem~\ref{thm:cong}. In other words, one does not obtain a stronger
  version of Theorem~\ref{thm:cong} by replacing condition \eqref{eq:cong:C}
  with the weaker condition \eqref{eq:cong:Cj2}.
  
  Recall that $S^{\ast} (p, r, s)$ is obtained from $S (p, r, s)$ by removing
  the element $s_0 \in \{ 0, 1, \ldots, p - 1 \}$ which is congruent to $s - r
  / 24$ modulo $p$. By construction, $s_0 = \operatorname{digit}_0 ( s - r / 24 ; p)$.
  Note that replacing $S (p, r, s)$ with $S^{\ast} (p, r, s)$ in
  Theorem~\ref{thm:cong} provides additional congruences only if $\max S ( p,
  r, s) = s_0$ and the value of $j$ is such that $j > p - 1 - s_0$. In this
  case, we have $p - j \leq s_0$, so that the second condition in
  \eqref{eq:cong:Cjx} does not hold.
\end{remark}

Finally, let us prove the considerably simpler Lemma~\ref{lem:cong:pr}.

\begin{proof}[Proof of Lemma~\ref{lem:cong:pr}]
  We have to show that $\xi_{p^{\lambda}
  r} ( p m + j) \equiv 0$ modulo $p^{\lambda}$ whenever $j$ is not divisible
  by $p$. These congruences are equivalent to the fact that, modulo
  $p^{\lambda}$, the generating series \eqref{eq:fishburn:rs} for the numbers
  $\xi_{p^{\lambda} r} ( n)$ can be written as
  \begin{equation*}
    \sum_{n \geq 0} \xi_{p^{\lambda} r} (n) q^n = \sum_{n \geq 0}
     \prod_{j = 1}^n (1 - (1 - q)^{p^{\lambda} r j}) \equiv \sum_{n \geq
     0} c_n q^{p n},
  \end{equation*}
  for some integers $c_n$. It therefore suffices to show that
  \begin{equation*}
    (1 - q)^{p^{\lambda} r j} \equiv h ( q^p) \pmod{p^{\lambda}},
  \end{equation*}
  for some power series $h ( q) \in \mathbb{Z} [ [ q]]$. After binomially
  expanding the left-hand side, this last congruence, follows if we can show
  that, for any integers $n$ and $k$, such that $k > 0$ and $p$ does not
  divide $k$,
  \begin{equation}
    \binom{p^{\lambda} n}{k} \equiv 0 \pmod{p^{\lambda}} .
    \label{eq:binom0k}
  \end{equation}
  Again, this is a direct consequence of Kummer's Theorem~\ref{thm:kummer}.
  Since the last $p$-adic digit of $k$ is nonzero by assumption, the $p$-adic
  addition of $k$ and $p^{\lambda} n - k$ has to involve at least $\lambda$
  carries in order for the sum $p^{\lambda} n$ to end in $\lambda$ many zero
  digits, and \eqref{eq:binom0k} follows.
\end{proof}

\section{Conclusion}

We have extended the congruences $\xi (p m - j) \equiv 0$ modulo $p$ of
Andrews and Sellers \cite{as-fish} as well as Garvan \cite{garvan-fish} to
the case of prime powers. Limited numerical evidence suggests that the
congruences for the Fishburn numbers provided by Theorem~\ref{thm:as:x} are
complete. That is, we expect that every congruence of the form
\begin{equation*}
  \xi ( \alpha m + \beta) \equiv 0 \pmod{\rho},
\end{equation*}
which holds for all integers $m$, is implied by Theorem~\ref{thm:as:x}
together with the Chinese Remainder Theorem. We have numerically verified this
conjecture for all $\beta < \alpha \leq 100$ and $\rho \leq 100$ by
checking that no additional congruences exist.

Finally, we echo the problem posed in \cite{garvan-fish} to investigate
congruences for coefficients related to other quantum modular forms, of which
$q^{1 / 24} F ( q)$, as defined in \eqref{eq:F}, is an instance
\cite{zagier-strange}.

\begin{acknowledgements}
  I am grateful to Frank Garvan for comments on an
earlier version of this paper.
\end{acknowledgements}

% \bibliography{../../math}
% \bibliographystyle{alphaabbr}

\end{document}